\documentclass[a4paper,11pt]{amsart}

       \usepackage{palatino, verbatim}
        \usepackage[latin1]{inputenc}
        \usepackage[T1]{fontenc}
        \usepackage{amsthm, amsfonts}
        \usepackage{amsfonts}
        \usepackage{graphicx}
        \usepackage{amssymb}
        \usepackage{amsmath}
        \usepackage{latexsym}
        \usepackage[all]{xy}

        \newtheorem{thm}{Theorem}[section]
          \newtheorem{cor}[thm]{Corollary}
          \newtheorem{lem}[thm]{Lemma}
          \newtheorem{prop}[thm]{Proposition}

        \theoremstyle{definition}
          
          \newtheorem{rem}{Remark}

          \newcommand\M{{\mathcal M}}
           
          \newcommand\E{{\mathcal E}}

           \newcommand\G{{\mathcal G}}
          \newcommand\im{\mathrm{Im}}
          \newcommand\Pic{\mathrm{Pic}}
          \newcommand\W{\mathcal W}
          \newcommand\oo{\mathcal O}
          
          \newcommand\Ext{\mathrm{Ext}}
          
          \newcommand\Hom{\mathrm{Hom}}
          \newcommand\Cliff{\mathrm{Cliff}}
          \newcommand\rk{\mathrm{rk}}

\title
{Green's Conjecture for curves on rational surfaces with an anticanonical pencil}
\author{Margherita Lelli--Chiesa}
\address{Humboldt Universit\"at zu Berlin, Institut f\"ur Mathematik, 10099 Berlin}
\email{lelli@math.hu-berlin.de}

\begin{document}
\begin{abstract}
Green's conjecture is proved for smooth curves $C$ lying on a rational surface $S$ with an anticanonical pencil, under some mild hypotheses on the line
bundle $L=\oo_S(C)$. Constancy of Clifford dimension, Clifford index and gonality of curves in the linear system $\vert L\vert$ is also obtained.
\end{abstract}
\maketitle
\section{Introduction}
Green's Conjecture concerning syzygies of canonical curves was first stated in \cite{koszul} and proposes a generalization of Noether's Theorem and the Enriques-Babbage Theorem in terms of Koszul cohomology, predicting that for a curve $C$
\begin{equation}\label{con}
K_{p,2}(C,\omega_C)=0\qquad \textrm{ if and only if }p<\Cliff(C).
\end{equation}
Quite remarkably, this would imply that the Clifford index of $C$ can be read off the syzygies of its canonical embedding. The implication $K_{p,2}(C,\omega_C)\neq0$ for $p\geq\Cliff(C)$ was immediately achieved by Green and Lazarsfeld (\cite[Appendix]{koszul}) and the conjectural part reduces to the vanishing $K_{c-1,2}(C,\omega_C)=0$ for $c=\Cliff(C)$, or equivalently, $K_{g-c-1,1}(C,\omega_C)=0$. 

One naturally expects the gonality $k$ of $C$ to be also encoded in the vanishing of some Koszul cohomology groups. In fact, Green-Lazarsfeld's Gonality Conjecture predicts that any line bundle $A$ on $C$ of sufficiently high degree satisfies
\begin{equation}\label{go}
K_{p,1}(C,A)=0\qquad\textrm{ if and only if }p\geq h^0(C,A)-k.
\end{equation}
Green (\cite{koszul}) and Ehbauer (\cite{ehbauer}) have shown that the statement holds true for any curve of gonality $k\leq 3$. As in the case of Green's Conjecture, one implication is well-known (cf. \cite[Appendix]{koszul}); it was proved by Aprodu (cf. \cite{marian1}) that the conjecture is thus equivalent to the existence of a non-special globally generated line bundle $A$ such that $K_{h^0(C,A)-k,1}(C,A)=0$.

Both Green's Conjecture and Green-Lazarsfeld's Gonality Conjecture are in their full generality still open. However, by specialization to curves on $K3$ surfaces, they were proved for a general curve in each gonality stratum of $M_g$  by Voisin and Aprodu (cf. \cite{evenv,oddv,marian}). Combining this with an earlier result of Hirschowitz and Ramanan (cf. \cite{ramanan}), the two conjectures follow for any curve of odd genus $g=2k-3$ and maximal gonality $k$. 

In \cite{marian}, Aprodu provided a sufficient condition for a genus $g$ curve $C$ of gonality $k\leq(g +2)/2$ to satisfy both conjectures; this is known as the \emph{linear growth condition} and is expressed in terms of the Brill-Noether theory of $C$ only:
\begin{equation}\label{mmm}
\dim W^1_{d}(C)\leq d-k\qquad \textrm{  for }k\leq d\leq g-k+2.
\end{equation}
Aprodu and Farkas (\cite{aprodu}) used the above characterization in order to establish Green's Conjecture for smooth curves lying on arbitrary $K3$ surfaces. It is natural to ask whether a similar strategy can solve Green's Conjecture for curves lying on anticanonical rational surfaces, since these share some common behaviour with $K3$ surfaces. The situation gets more complicated because such a surface $S$ is in general non-minimal and its canonical bundle is non-trivial; in particular, given a line bundle $L\in \Pic(S)$, smooth curves in the linear system $\vert L\vert$ do not form a family of curves with constant syzygies, as it happens instead in the case of $K3$ surfaces. Our main result is the following:
\begin{thm}\label{veloce}
Let $S$ be a smooth, projective, rational surface with an anticanonical pencil and let $L$ be a line bundle on $S$ such that $L\otimes \omega_S$ is nef and big. In the special case where $h^0(S,\omega_S^\vee)=\chi(S,\omega_S^\vee)=2$, also assume that the Clifford index of a general curve in $\vert L\vert$ is not computed by the restriction of the anticanonical bundle $\omega_S^\vee$.

Then, any smooth, irreducible curve $C\in\vert L\vert$ satisfies Green's Conjecture.
\end{thm}
With no hypotheses on the line bundle $L$, we obtain Green's Conjecture and Green-Lazarsfeld's Gonality Conjecture for a general  curve in $\vert L\vert_s$, where $\vert L\vert_s$ denotes the locus of smooth and irreducible curves in the linear system $\vert L\vert$  (cf. Proposition \ref{general}). For later use, we denote by $g(L):=1+(c_1(L)^2+c_1(L)\cdot K_S)/2$ the genus of any curve in $\vert L\vert_s$.

Examples of surfaces as in Theorem \ref{veloce} are given by all rational surfaces $S$ whose canonical divisor satisfies $K_S^2>0$, or equivalently, having Picard number $\rho(S)\leq 9$, such as Del Pezzo surfaces ($-K_S$ is ample), generalized Del Pezzo surfaces ($-K_S$ is nef and big), some blow-ups of Hirzebruch surfaces. However, the class of surfaces that we are considering also includes surfaces $S$ with $K_S^2\leq 0$, such as rational elliptic surfaces (i.e., smooth complete complex surfaces that can be obtained  by blowing up $\mathbb{P}^2$ at $9$ points, which are the base locus of a pencil of cubic curves with at least one smooth member). 

We also obtain the following:
\begin{thm}\label{cli}
Assume the same hypotheses as in Theorem \ref{veloce} and let $g(L)\geq 4$. Then, all curves in $\vert L\vert_s$ have the same Clifford dimension $r$, the same Clifford index and the same gonality. Moreover, if the curves in $\vert L\vert_s$ are exceptional, then one of the following occurs:
\begin{itemize}
\item[(i)] $r=2$ and any curve in $\vert L\vert_s$ is the strict transform of a smooth, plane curve under a morphism $\phi:S\to\mathbb{P}^2$ which is the composition of finitely many blow-ups.
\item[(ii)] $r=3$ and $S$ can be realized as the blow-up of a normal cubic surface $S'\subset \mathbb{P}^3$ at a finite number of points (possibly infinitely near); any curve in $\vert L\vert_s$ is the strict transform under the blow-up map of a smooth curve in $\vert -3 K_{S'}\vert$.
\end{itemize}
\end{thm}
This generalizes results of Pareschi (cf. \cite{delpezzo}) and Knutsen (cf. \cite{leopold}) concerning the Brill-Noether theory of curves lying on a Del Pezzo surface $S$. In \cite{leopold}, the author proved that line bundles violating the constancy of the Clifford index only exist when $K_S^2=1$; they are described  in terms of the coefficients of the generators of $\Pic(S)$ in their presentation. In fact, one can show that such line bundles are exactly those satisfying: $L\otimes\omega_S$ is nef and big and the restriction of the anticanonical bundle $\omega_S^\vee$ to a general curve in $\vert L\vert_s$ computes its Clifford index (cf. Remark \ref{finale}).

The proofs of Theorem \ref{veloce} and Theorem \ref{cli} rely on vector bundle techniques \`a la Lazarsfeld (cf. \cite{lazarsfeld}); in particular, we consider rank-$2$ bundles $E_{C,A}$, which are the analogue of the Lazarsfeld-Mukai bundles for $K3$ surfaces. The key fact proved in Section \ref{to} is that, if $A$ is a complete, base point free pencil on a general curve $C\in\vert L\vert_s$, the dimension of $\ker\mu_{0,A}$ is controlled by $H^2(S,E_{C,A}\otimes E_{C,A}^\vee)$; if this is nonzero, the bundle $E_{C,A}$ cannot be slope-stable with respect to any polarization $H$ on $S$. 

By considering Harder-Narasimhan and Jordan-H\"older filtrations, in Section \ref{ne} we perform a parameter count for pairs $(C,A)$ such that $E_{C,A}$ is not $\mu_H$-stable; this gives an upper bound for the dimension of any irreducible component $\W$ of $\W^1_d(\vert L\vert)$ which dominates $\vert L\vert$ under the natural projection $\pi:\W^1_d(\vert L\vert)\to\vert L\vert_s$. It turns out (cf. Proposition \ref{cliffdim}) that, if a general curve $C\in\vert L\vert_s$ is exceptional, the same holds true for all curves in $\vert L\vert_s$ and one is either in case (i) or (ii) of Theorem \ref{cli}; in this context we recall that Green's Conjecture for curves of Clifford dimension $2$ and $3$ was verified by Loose in \cite{loose}. If instead $C$ has Clifford dimension $1$, our parameter count ensures that it satisfies the linear growth condition (\ref{mmm}). In order to deduce Green's Conjecture for \emph{every} curve in $\vert L\vert_s$, we make use of the hypotheses made on $L$ and show that the Koszul group $K_{g-c-1,1}(C,\omega_C)$ does not depend (up to isomorphism) on the choice of $C\in\vert L\vert_s$, as soon as $c$ equals the Clifford index of a general curve in $\vert L\vert_s$. Semicontinuity will imply the constancy of the Clifford index and the gonality.\\\textbf{Acknowledgements:} I am grateful to my advisor Gavril Farkas, who suggested me to investigate the topic.

\section{Syzygies and Koszul Cohomology}\label{ge}
If $L$ is an ample line bundle on a complex projective variety $X$, let $S:=\mathrm{Sym}^\ast H^0(X,L)$ be the homogeneous coordinate ring of the projective space $\mathbb{P}(H^0(X,L)^\vee)$ and set $R(X):=\bigoplus_m H^0(X,L^{m})$. Being a finitely generated $S$-module, $R(X)$ admits a minimal graded free resolution
$$
0\to E_{s}\to\ldots\to E_1\to E_0\to R(X)\to 0,
$$
where for $k\geq 1$ one can write $E_k=\sum_{i\geq k}S(-i-1)^{\beta_{k,i}}$. The \emph{syzygies} of $X$ of order $k$ are by definition the graded components of the $S$-module $E_k$. We say that the pair $(X,L)$ satisfies property $(N_p)$ if $E_0=S$ and $E_k=S(-k-1)^{\beta_{k,k}}$ for all $1\leq k\leq p$. In other words, property $(N_0)$ is satisfied whenever $\phi_L$ embeds $X$ as a projectively normal variety, while property $(N_1)$ also requires that that the ideal of $X$ in $\mathbb{P}(H^0(X,L)^\vee)$ is generated by quadrics; for $p>1$, property $(N_p)$ means that the syzygies of $X$ up to order $p$ are linear. 

The most effective tool in order to compute syzygies is Koszul cohomology, whose definition is the following. Let $L\in\Pic(X)$ and $F$ be a coherent sheaf on $X$. The Koszul cohomology group $K_{p,q}(X,F,L)$ is defined as the cohomology at the middle-term of the complex 
$$
\bigwedge^{p+1} H^0(L)\otimes H^0(F\otimes L^{q-1})\to\bigwedge^{p} H^0(L)\otimes H^0(F\otimes L^{ q})\to\bigwedge^{p-1} H^0(L)\otimes H^0(F\otimes L^{q+1}).
$$
When $F\simeq \oo_X$, the Koszul cohomology group is conventionally denoted by $K_{p,q}(X,L)$. It turns out (cf. \cite{koszul}) that property $(N_p)$ for the pair $(X,L)$ is equivalent to the vanishing
$$
K_{i,q}(X,L)=0\qquad \textrm{  for all }i\leq p\textrm{ and }q\geq 2.
$$
In particular, Green's Conjecture can be rephrased by asserting that $(C,\omega_C)$ satisfies property $(N_p)$ whenever $p<\Cliff(C)$.

In the sequel we will make use of the following results, which are due to Green. The first one is the Vanishing Theorem (cf. \cite[Theorem (3.a.1)]{koszul}), stating that
\begin{equation}\label{vanishing}
K_{p,q}(X,E,L)=0\qquad \textrm{  if  }p\geq h^0(X,E\otimes L^{ q}).
\end{equation}
The second one (cf. \cite[Theorem (3.b.1)]{koszul}) relates the Koszul cohomology of $X$ to that of a smooth hypersurface $Y\subset X$ in the following way.
\begin{thm}\label{complex}
Let $X$ be a smooth irreducible projective variety and assume $L, N\in\Pic(X)$ satisfy 
\begin{eqnarray}
\label{sa}H^0(X,N\otimes L^\vee)&=&0\\ 
\label{sso}H^1(X,N^{ q}\otimes L^\vee)&=&0,\qquad \textrm{   }\forall\, q\geq 0.
\end{eqnarray}
 Then, for every smooth integral divisor $Y\in\vert L\vert$, there exists a long exact sequence 
$$
\cdots \to K_{p,q}(X,L^\vee,N)\to K_{p,q}(X,N)\to K_{p,q}(Y,N\otimes\oo_Y)\to K_{p-1,q+1}(X,L^\vee,N)\to\cdots.
$$
\end{thm}

\section{Petri map via vector bundles}\label{to}
Let $S$ be a smooth rational surface with an anticanonical pencil and $C\subset S$ be a smooth, irreducible curve of genus $g$. We set $L:=\oo_S(C)$. If $A$ is a complete, base point free $g^r_d$ on $C$, as in the case of $K3$ surfaces, let $F_{C,A}$ be the vector bundle on $S$ defined by the sequence
$$
0\to F_{C,A}\to H^0(C,A)\otimes \oo_S\stackrel{ev_{A,S}}{\longrightarrow} A\to 0,$$
and set $E_{C,A}:=F_{C,A}^\vee$. Since $N_{C\vert S}=\oo_C(C)$, by dualizing the above sequence we get
\begin{equation}\label{bonn}
0\to H^0(C,A)^\vee\otimes \oo_S \to E_{C,A}\to \oo_C(C)\otimes A^\vee\to 0.
\end{equation}
This trivially implies that:
\begin{itemize}
\item $\chi(S,E_{C,A}\otimes \omega_S)=h^0(S,E_{C,A}\otimes \omega_S)=g-d+r$, 
\item $\rk\, E_{C,A}=r+1$, $c_1(E_{C,A})=L$, $c_2(E)=d$,
\item $h^2(S,E_{C,A})=0$, $\chi(S,E_{C,A})=g-d+r-c_1(L)\cdot K_S$.
\end{itemize}
Being a bundle of type $E_{C,A}$ is an open condition. Indeed, a vector bundle $E$ of rank $r+1$ is of type $E_{C,A}$ whenever $h^1(S,E\otimes\omega_S) = h^2(S,E\otimes\omega_S) = 0$ and there exists $\Lambda\in G(r + 1, H^0(S, E))$ such that the degeneracy locus of the evaluation map $ev_\Lambda : \Lambda\otimes\oo_S\to E$ is a smooth connected curve.

Notice that the dimension of the space of global sections of $E_{C,A}$ depends not only on the type of the linear series $A$ but also on $A\otimes \omega_S$. In particular, one has 
\begin{eqnarray*}
h^0(S,E_{C,A})&=&r+1+h^0(C, \oo_C(C)\otimes A^\vee),\\
h^1(S,E_{C,A})&=&h^0(C,A\otimes\omega_S).
\end{eqnarray*}
Moreover, if the line bundle $\oo_C(C)\otimes A^\vee$ has sections, then $E_{C,A}$ is generated off its base points.
In the case $r=1$, we prove the following.
\begin{lem}\label{impo}
Let $A$ be a complete, base point free $g^1_d$ on $C\subset S$. If either
\begin{itemize}
\item $h^0(S,\omega_S^\vee)>2$, or
\item $h^0(S,\omega_S^\vee)=2$ and $A\not\simeq \omega_{S}^\vee\otimes\oo_C$
\end{itemize}
holds, then $h^0(C,A\otimes\omega_S)=0$.
\end{lem}
\begin{proof}
Since $L\otimes \omega_S$ is effective, the short exact sequence
$$
0\to L^\vee\otimes\omega_S^\vee\to \omega_S^\vee\to \omega_{S}^\vee\otimes \oo_C\to 0$$
implies $h^0(C,\omega_{S}^\vee\otimes \oo_C)\geq h^0(S,\omega_S^\vee)$ and the statement follows trivially if $ h^0(S,\omega_S^\vee)>2$. Let $ h^0(S,\omega_S^\vee)=2$ and $h^0(C,A\otimes\omega_S)>0$. Then necessarily $h^0(C,\omega_{S}^\vee\otimes \oo_C)=2$ and $A\otimes\omega_S$ is the fixed part of the linear system of sections of $A$. Since $A$ is base point free by hypothesis, then $A\simeq\omega_{S}^\vee\otimes \oo_C$.
\end{proof}
Under the hypotheses of the above Lemma, the bundle $E_{C,A}$ is globally generated off a finite set and $\chi(S,E_{C,A})=h^0(S,E_{C,A})=g-d+1-c_1(L)\cdot K_S$. We remark that the vanishing of $h^1(S,E_{C,A})$ turns out to be crucial in most of the following arguments and this is why the assumptions on the anticanonical linear system of $S$ are needed.

The following proposition gives a necessary and sufficient condition for the injectivity of the Petri map $\mu_{0,A}:H^0(C,A)\otimes H^0(C,\omega_C\otimes A^\vee)\to H^0(C,\omega_C)$. 
\begin{prop}\label{seppia}
If $C\in\vert L\vert_s$ is general and either $h^0(S,\omega_S^\vee)>2$ or $h^0(S,\omega_S^\vee)=2$ and $A\not\simeq \omega_{S}^\vee\otimes\oo_C$, then for any complete, base point free pencil $A$ on $C$ one has:
$$
\ker\mu_{0,A}\simeq H^2(S,E_{C,A}\otimes E_{C,A}^\vee).
$$
In particular, the vanishing of the one side implies the vanishing of the other.
\end{prop}
\begin{proof}
The proof proceeds as in \cite{pareschi}, hence we will not enter into details. As $A$ is a pencil, the kernel of the evaluation map $ev_{A,C}:H^0(C,A)\otimes \oo_C\twoheadrightarrow A$ is ismomorphic to $A^\vee$ and $\ker\mu_{0,A}\simeq H^0(C,\omega_C\otimes A^{-2})$. Since $\det F_{C,A}=L^\vee$, by adjunction one finds the following short exact sequence:
\begin{equation}\label{prr}0\to \omega_{S}\otimes\oo_C\to F_{C,A}\otimes \omega_C\otimes A^\vee\to \omega_C\otimes A^{-2}\to 0.\end{equation}
The coboundary map $\delta: H^0(C, \omega_C\otimes A^{-2})\to H^1(C,\omega_{S}\otimes\oo_C)$ coincides, up to multiplication by a nonzero scalar factor, with the composition of the Gaussian map
$$\mu_{1,A}:\ker\mu_{0,A}\to H^0(C,\omega_C^{ 2})$$
and the dual of the Kodaira spencer map
$$
\rho^\vee:  H^0(C,\omega_C^{ 2})\to (T_C\vert L\vert)^\vee=H^0(C,N_{C\vert S})^\vee=H^1(C,\omega_S\otimes\oo_C).$$
Indeed, as in \cite[Lemma 1]{pareschi}, one finds a commutative diagram
$$
\xymatrix{
0\ar[r]&\omega_{S}\otimes\oo_C\ar[r]\ar@{=}[d]&F_{C,A}\otimes \omega_C\otimes A^\vee\ar[r]\ar[d]&\omega_C\otimes A^{-2}\ar[r]\ar[d]^s&0\\
0\ar[r]&\omega_{S}\otimes\oo_C\ar[r]&\Omega^1_S\otimes \omega_C\ar[r]&\omega_C^{ 2}\ar[r]&0,\\
}
$$
where the homomorphism induced by $s$ on global sections is $\mu_{1,A}$ and the coboundary map $H^0(C,\omega_C^{ 2})\to H^1(C,\omega_{S}\otimes\oo_C)$ equals (up to a scalar factor) $\rho^\vee$. 

If $A$ has degree $d$, look at the natural projection $\pi:\W^1_d(\vert L\vert)\to\vert L\vert_s$. First order deformation arguments (see, for instance, \cite[p. 722]{arbarello2}) imply that 
$$\im(d\pi_{(C,A)})\subset \mathrm{Ann}(\im(\rho^\vee\circ \mu_{1,A})).$$
Therefore, by Sard's Lemma, if $C\in\vert L\vert_s$ is general, the short exact sequence (\ref{prr}) is exact on the global sections for any base point free $A\in W^1_d(C)\setminus W^2_d(C)$, and $\ker\mu_{0,A}\simeq H^0(C,F_{C,A}\otimes \omega_C\otimes A^\vee)$. By tensoring short exact sequence (\ref{bonn}) with $F_{C,A}\otimes \omega_S$, one finds that $$H^0(C,F_{C,A}\otimes \omega_C\otimes A^\vee)\simeq H^0(S,E_{C,A}^\vee\otimes E_{C,A}\otimes \omega_S)$$ because $H^i(S,F_{C,A}\otimes \omega_S)\simeq H^{2-i}(S,E_{C,A})^\vee=0$ for $i=0,1$. The statement follows now by Serre duality.
\end{proof}
\begin{cor}\label{espresso}
Let $H$ be any polarization on $S$ and $\W$ be an irreducible component of $\W^1_d(\vert L\vert)$ which dominates $\vert L\vert$ and whose general points correspond to $\mu_{H}$-stable bundles; in the special case where $h^0(S,\omega_S^\vee)=2$, assume that general points of $\W$ are not of the form $(C, \omega_{S}^\vee\otimes \oo_C)$.

Then, $\rho(g,1,d)\geq 0$ and $\W$ is reduced of dimension equal to $\dim\vert L\vert+\rho(g,1,d)$.
\end{cor}
\begin{proof}
Let $(C,A)$ be a general point of $\W$. If $E_{C,A}$ is stable, $E_{C,A}\otimes\omega_S$ also is. The inequality $\mu_{H}(E_{C,A})>\mu_{H}(E_{C,A}\otimes \omega_S)$ implies that $H^2(S,E_{C,A}^\vee\otimes E_{C,A})\simeq\Hom(E_{C,A},E_{C,A}\otimes\omega_S)^\vee=0$. As a consequence, $\W$ is smooth in $(C,A)$ of the expected dimension.
\end{proof}
\begin{rem}\label{gusto}
Corollary \ref{espresso} can be alternatively proved by arguing in the following way.\\ Let $M:=M^{\mu s}_{H}(c)$ be the moduli space of $\mu_{H}$-stable vector bundles on $S$ of total Chern class $c=2+c_1(L)+d\omega\in H^{2 *}(S,\mathbb{Z})$, where $\omega$ is the fundamental cocycle. Since every $[E]\in M$ satisfies $\Ext^2(E,E)_0=0$, it turns out that $M$ is a smooth, irreducible projective variety of dimension $4d-c_1(L)^2-3$ (cf. \cite[Remark 2.3]{costa}), as soon as it is non-empty. Let $M^0$ be the open subset of $M$ parametrizing vector bundles $[E]$ of type $E_{C,A}$; if this is nonempty, define $\G$ as the Grassmann bundle on $M^0$ with fiber over $[E]$ equal to $G(2,H^0(S,E))$. Look at the rational map $h:\G\dashrightarrow \W^1_d(\vert L\vert)$ sending a general $(E,\Lambda)\in\G$ to the pair $(C_\Lambda,A_\Lambda)$, where $C_\Lambda$ is the degeneracy locus of the evaluation map $ev_\Lambda:\Lambda\otimes\oo_S\to E$ and $\oo_{C_\Lambda}(C_\Lambda)\otimes A^\vee$ is its cokernel. Since any $[E]\in M^0$ is simple, one easily checks that $h$ is birational onto its image, that is denoted by $\W$. As a consequence, the dimension of $\W$ equals: $$4d-c_1(L)^2-3+2(g-d-1-c_1(L)\cdot K_S)=2d-3-c_(L)\cdot K_S\leq \dim\vert L\vert +\rho(g,1,d).$$
\end{rem}

\section{Parameter count}\label{ne}
By the analysis performed in the previous section, given a polarization $H$ on $S$, the linear growth condition for a general curve in $\vert L\vert_s$ can be verified by controlling the dimension of every dominating component $\W\subset \W^1_d(\vert L\vert)$, whose general points are pairs $(C,A)$ such that $A\not\simeq\omega_{S}^\vee\otimes \oo_C$ and the bundle $E_{C,A}$ is not $\mu_{H}$-stable. Indeed, if $A\simeq\omega_{S}^\vee\otimes \oo_C$ for a general point of $\W$, then $\omega_{S}^\vee\otimes \oo_{C'}$ is an isolated point of $W^1_d(C')$ for every $C'\in\vert L\vert_s$.

Let $A$ be a complete, base point free $g^1_d$ on a curve $C\in\vert L\vert_s$ such that the bundle $E:=E_{C,A}$ is not $\mu_{H}$-stable and $A\not\simeq\omega_{S}^\vee\otimes \oo_C$ if $h^0(S,\omega_S^\vee)=2$. The maximal destabilizing sequence of $E$ has the form
\begin{equation}\label{uova}
0\to M\to E\to N\otimes I_\xi\to 0,
\end{equation}
where $M, N\in\Pic(S)$ satisfy 
\begin{equation}\label{slope}
\mu_{H}(M)\geq \mu_H(E)\geq \mu_{H}(N),
\end{equation}
with both or none of the  inequalities being strict, and $I_\xi$ is the ideal sheaf of a $0$-dimensional subscheme $\xi\subset S$ of length $l=d-c_1(N)\cdot c_1(M)$.
\begin{lem}\label{ll}
In the above situation, assume that general curves in $\vert L\vert_s$ have Clifford index $c$. If $\mu_{0,A}$ is non-injective and $C$ is general in $\vert L\vert$, then the following inequality holds:
\begin{equation}\label{yuppi}
c_1(M)\cdot c_1(N)+c_1(N)\cdot K_S\geq c.
\end{equation}
\end{lem}
\begin{proof}
Being a quotient of $E:=E_{C,A}$ off a finite set, $N$ is base component free and is non-trivial since $H^2(S,N\otimes \omega_S)=0$. As a consequence, $h^0(S,N)\geq 2$. Proposition \ref{seppia} implies that $\Hom(E,E\otimes\omega_S)\neq 0$. Applying $\Hom(E,-)$ to the short exact sequence (\ref{uova}) twisted with $\omega_S$, we obtain
$$
0\to\Hom(E,M\otimes\omega_S)\to \Hom(E,E\otimes\omega_S)\to \Hom(E,N\otimes\omega_S\otimes I_\xi)\to\cdots.
$$
Apply now $\Hom(-,N\otimes\omega_S\otimes I_\xi)$ (respectively $\Hom(-,M\otimes\omega_S)$) to exact sequence (\ref{uova}), and find that $\Hom(E,N\otimes\omega_S\otimes I_\xi)=0$ (resp. $\Hom(E,M\otimes\omega_S)\simeq \Hom(N\otimes I_\xi,M\otimes\omega_S)$), hence $N^\vee\otimes M\otimes\omega_S$ is effective and $h^0(S,M\otimes \omega_S)\geq 2$. This ensures that $N\otimes\oo_C$ contributes to the Clifford index of $C$ and 
\begin{eqnarray*}
c\leq\Cliff(N\otimes\oo_C)&=&c_1(N)\cdot (c_1(N)+c_1(M))-2h^0(C,N\otimes\oo_C)+2\\
&\leq& c_1(N)^2+c_1(N)\cdot c_1(M)-2h^0(S,N)+2\\
&=&c_1(N)\cdot c_1(M)+c_1(N)\cdot K_S.
\end{eqnarray*}

\end{proof}
Now, upon fixing a nonnegative integer $l$ and a line bundle $N$ such that (\ref{slope}) is satisfied for $M:=L\otimes N^\vee$, we want to estimate the number of moduli of pairs $(C,A)$ such that the bundle $E_{C,A}$ sits in a short exact sequence like (\ref{uova}). The following construction is analogous to the one performed in \cite[Section 4]{ciotola} in the case of $K3$ surfaces. Let $\mathcal{E}_{N,l}$ be the moduli stack of extensions of type (\ref{uova}), where $l(\xi)=l$.
Having fixed $c\in H^{2*}(S,\mathbb{Z})$, we denote by $\M(c)$ the moduli stack of coherent sheaves of total Chern class $c$. We consider the natural maps $p:\E_{N,l}\to \M(c(M))\times \M(c(N\otimes I_\xi))$ and $q:\E_{N,L}\to\M(c(E))$, which send the $\mathbb{C}$-point of  $\E_{N,l}$ corresponding to extension (\ref{uova}) to the classes of $(M,N\otimes I_\xi)$ and $E$ respectively.
Notice that, since $M, N$ lie in $\Pic(S)$, the stack $\M(c(M))$ has a unique $\mathbb{C}$-point endowed with a $1$-dimensional space of automorphisms, while $\M(c(N\otimes I_\xi))$ is corepresented by the Hilbert scheme $S^{[l]}$ parametrizing $0$-dimensional subschemes of $S$ of length $l$.

We denote by $\tilde{Q}_{N,l}$ the closure of the image of $q$ and by $Q_{N,l}$ its open substack consisting of vector bundles of type $E_{C,A}$ for some $C\in\vert L\vert_s$ and $A\in W^1_d(C)$, with $d:=l+c_1(M)\cdot c_1(N)$ and $A\not\simeq \omega_S^\vee\otimes  \oo_C$ if $h^0(S,\omega_S^\vee)=2$. Let $\G_{N,l}\to Q_{N,l}$ be the Grassmann bundle whose fiber over $[E]\in Q_{N,l}(\mathbb{C})$ is $G(2,H^0(S,E))$. By construction, a $\mathbb{C}$-point of $\G_{N,l}$ corresponding to a pair $([E],\Lambda)$, with $\Lambda\in G(2,H^0(S,E))$, comes endowed with an automorphism group equal to $\mathrm{Aut}(E)$. We define $\W_{N,l}$ to be the closure of the image of the rational map $\G_{N,l}\dashrightarrow \W^1_d(\vert L\vert)$, which sends a general point $([E],\Lambda)\in\G_{N,l}(\mathbb{C})$ to the pair $(C_\Lambda,A_\Lambda)$ where the evaluation map $ev_\Lambda:\Lambda\otimes\oo_S\hookrightarrow E$ degenerates on $C_\Lambda$ and has $\oo_{C_\Lambda}(C_\Lambda)\otimes A_\Lambda^\vee$ as cokernel. The following proposition gives an upper bound for the dimension of $\W_{N,l}$.
\begin{prop}\label{conto}
Assume that general curves in $\vert L\vert_s$ have Clifford index $c$. Then, every irreducible component $\W$ of $\W^1_d(\vert L\vert_s)$ which dominates $\vert L\vert$ and is contained in $\W_{N,l}$ satisfies
$$
\dim\W\leq \dim\vert L\vert +d-c-2.
$$
\end{prop} 
\begin{proof}
The fiber of $p$ over a $\mathbb{C}$-point of $\M(c(M))\times \M(c(N\otimes I_\xi))$ corresponding to $(M,N\otimes I_\xi)$ is the quotient stack 
$$
[\Ext^1(N\otimes I_\xi,M)/\Hom(N\otimes I_\xi,M)],
$$
where the action of the Hom group on the Ext group is trivial. The fiber of $q$ over $[E]\in \tilde{Q}_{N,l}(\mathbb{C})$ is the Quot-scheme $\mathrm{Quot}_S(E,P)$, where $P$ is the Hilbert polynomial of $N\otimes I_\xi$. The condition $\mu_{H}(M)\geq \mu_{H}(N)$ implies that $\Ext^2(N\otimes I_\xi,M)\simeq \Hom(M, N\otimes \omega_S\otimes I_\xi)^\vee=0$, hence the dimension of the fibers of $p$ is constant and equals 
$$
-\chi(S,N\otimes M^\vee\otimes \omega_S\otimes I_\xi)=-g+2c_1(N)\cdot c_1(M)+c_1(M)\cdot K_S+l.$$
By looking at the tangent and obstruction spaces at any point, one shows that the Quot schemes constituting the fibers of $q$ are either all $0$-dimensional or all smooth of dimension $1$; indeed, $\Hom(M, N\otimes I_\xi)=0$ unless $M\simeq N$ and $l=0$, in which case $\Ext^1(M, N\otimes I_\xi)=H^1(S,\oo_S)=0$. As a consequence, if nonempty, $Q_{N,l}$ has dimension at most $3l-2-g+2c_1(N)\cdot c_1(M)+c_1(M)\cdot K_S$.

Since the map $h_{N,l}$ forgets the automorphisms, its fiber over a pair $(C,A)\in\W_{N,l}$ is the quotient stack 
$$[U/\mathrm{Aut}(E_{C,A})],$$
where $U$ is the open subscheme of $\mathbb{P}(\Hom(E_{C,A},\oo_C(C)\otimes A^\vee))$ whose points correspond to morphisms with kernel equal to $\oo_S^{\oplus 2}$, and $\mathrm{Aut}(E_{C,A})$ acts on $U$ by composition. Using the vanishing $h^i(S,E_{C,A}\otimes \omega_S)=0$ for $i=1,2$, one checks that $$\Hom(E_{C,A},\oo_C(C)\otimes A^\vee)\simeq H^0(S,E_{C,A}\otimes E_{C,A}^\vee),$$ and $U$ is isomorphic to $\mathbb{P}\mathrm{Aut}(E_{C,A})$. Hence, the fibers of $h_{N,l}$ are stacks of dimension $-1$ and
\begin{eqnarray*}
\dim\W_{N,l}&\leq&3l-1-g+2c_1(N)\cdot c_1(M)+c_1(M)\cdot K_S+2(g-d-1-c_1(L)\cdot K_S)\\
&=&d+g-3-c_1(N)\cdot c_1(M)-c_1(N)\cdot K_S-c_1(L)\cdot K_S.
\end{eqnarray*}
The conclusion follows now from the fact that $\dim\vert L\vert\geq g-1-c_1(L)\cdot K_S$, along with Lemma \ref{ll}.
\end{proof}

\section{Proof of the main results}\label{eee}
We recall some facts about exceptional curves, that is, curves of Clifford dimension greater than $1$. Coppens and Martens (\cite{coppi}) proved that, if $C$ is an exceptional curve of gonality $k$ and Clifford dimension $r$, then $\Cliff(C)=k-3$ and $C$ possesses  a $1$-dimensional family of $g^1_k$. Furthermore, if $r\leq 9$, there exists a unique line bundle computing $\Cliff(C)$ (cf. \cite{cliff}); this is conjecturally true for any $r$. Curves of Clifford dimension $2$ are precisely the smooth plane curves of degree $\geq 5$, while the only curves of Clifford dimension $3$ are the complete intersections of two cubic surfaces in $\mathbb{P}^3$ (cf. \cite{marti}). We will use these results in the proof of the following:
\begin{prop}\label{cliffdim}
Let $L$ be a line bundle on a smooth, rational surface $S$ with an anticanonical pencil. If $g(L)\geq 4$ and a general curve $C\in \vert L\vert_s$ is exceptional, then any other curve inside $\vert L\vert_s$ has the same Clifford dimension $r$ of $C$ and either case (i) or (ii) of Theorem \ref{cli} occurs.
\end{prop}
\begin{proof}
Since any curve of odd genus and maximal gonality has Clifford dimension $1$ (cf. \cite[Corollary 3.11]{marian2}), we can assume that general curves in $\vert L\vert_s$ have gonality $k\leq (g+2)/2$ and are exceptional. There exists a component $\W$ of $\W^1_k(\vert L\vert)$ of dimension at least $\dim\vert L\vert+1$ and, by Corollary \ref{espresso}, this is contained in $\W_{N,l}$ for some $N$ and $l$. Notice that the line bundle $N$ is nef since it is globally generated off a finite set. Furthermore, it follows from the proof of Proposition \ref{conto} that $N$ and $M:=L\otimes N^\vee$ satisfy equality in (\ref{yuppi}), that is, 
$$
k-3=c_1(M)\cdot c_1(N)+c_1(N)\cdot K_S=k-l+c_1(N)\cdot K_S;
$$ 
in particular, $N\otimes\oo_C$ computes the Clifford index of a general $C\in\vert L\vert_s$ and $h^1(S,M^\vee)=0$. Having at least a $2$-dimensional space of sections, the line bundle $\omega_S^\vee\otimes \oo_C$ has degree $\geq k$, thus $-c_1(M)\cdot K_S\geq k-3+l$. 

Assume $h^0(S,N\otimes \omega_S)\geq 2$; the restriction of $M$ to a general curve $C\in\vert L\vert_s$ contributes to its Clifford index and 
$$
k-3\leq \Cliff(M\otimes\oo_C)=c_1(M)\cdot c_1(N)+c_1(M)\cdot K_S\leq 3-2l.
$$
As $k\geq 2r$ (cf. \cite[Proposition 3.2]{cliff}), we have $r\leq 3$; if $r=3$, then $l=0$, while $r=2$ implies $l\leq 1$. Let $r=2$; since $\chi(S,N)=h^0(S,N)=h^0(C,N\otimes \oo_C)=3$ and $h^i(S,N\otimes \omega_S)=0$ for $i=1,2$ (as one can check twisting (\ref{uova}) with $\omega_S$ and taking cohomology), then $c_1(N)^2=l+1$ and $h^0(S,N\otimes \omega_S)=l\leq 1$, contradicting our assumption. Hence, the inequality $h^0(S,N\otimes \omega_S)\geq 2$ implies $r=3$ and $l=0$.
  
Assume instead that $h^0(S,N\otimes \omega_S)\leq 1$; we get $c_1(N)^2\leq 3-l$ and $h^0(C,N\otimes \oo_C)=h^0(S,N)=\chi(S,N)\leq 4-l$. Since $N\otimes \oo_C$ computes the Clifford index of $C$, then $r\leq 3$ holds in this case as well. Moreover, $l=0$ when $r=3$, and $l\leq 1$ if $r=2$.

Let $r=2$ and $l=1$; then, $c_1(N)^2=-c_1(N)\cdot K_S=2$. By \cite[Lemma 2.6, Theorem 2.11]{birational}, $N$ is base point free and not composed with a pencil, hence it defines a generically $2:1$ morphism $\phi:=\phi_N:S\to\mathbb{P}^2$ splitting into the composition of a birational morphism $\psi:S\to S'$, which contracts the finitely many curves $E_1,\cdots,E_h$ having zero intersection with $c_1(N)$, and a ramified double cover $\pi:S'\to\mathbb{P}^2$. Set $N':=\pi^*(\oo_{\mathbb{P}^2}(1))$; since $N=\psi^*(N')$ and $\psi^*$ preserves both the intersection products and the dimensions of cohomology groups, we have $c_1(N')^2=-c_1(N')\cdot K_{S'}=2$ and
$$
1=h^0(S,N\otimes\omega_S)\geq h^0(S,N\otimes\omega_S(-E_1-\cdots-E_h))=h^0(S',N'\otimes\omega_{S'}).
$$
We apply Theorem 3.3. in \cite{birational} and get $N'=\omega_{S'}^\vee$ and $K_{S'}^2=2$ (cases (b), (c), (d) of the aforementioned theorem cannot occur since they would contradict $c_1(N')^2=2$). The line bundle $N\otimes \oo_C$ is very ample because it computes $\Cliff(C)$ (cf. \cite[Lemma 1.1]{cliff}); hence, $C$ is isomorphic to $C'=\psi(C)$ and $\omega_{S'}^\vee\otimes\oo_{C'}$ is also very ample. Proceeding as in the proof of \cite[Lemma 2.6]{delpezzo} (where the ampleness of $\omega_{S'}^\vee$ is not really used), one shows that $\phi(C')\in\vert-2K_{S'}\vert$. This gives a contradiction because it implies $g(C')=g(C)=3$.

Up to now, we have shown that $r\leq 3$ and $l=0$, hence $-c_1(N)\cdot K_S=3$ and $c_1(N)^2>0$. By \cite[Proposition 3.2]{birational}, the line bundle $N$ defines a morphism $\phi_N:S\to \mathbb{P}^r$ which is birational to its image and only contracts the finitely many curves which have zero intersection with $c_1(N)$. If $r=2$, then $\phi_N$ is the blow-up of $\mathbb{P}^2$ at finitely many points (maybe infinitely near) and any curve in $\vert L\vert_s$ is the strict transform of a smooth plane curve. For $r=3$, the image of $\phi_N$ is a normal cubic surface $S'\subset \mathbb{P}^3$ and any curve in $\vert L\vert_s$ is the strict transform of a smooth curve in $\vert -3 K_{S'}\vert$, hence has Clifford dimension $3$.
\end{proof}
The following result is now straightforward.
\begin{prop}\label{general}
Let $C$ be a smooth, irreducible curve lying on a rational surface $S$ with an anticanonical pencil. If $C$ is general in its linear system, then $C$ satisfies Green's Conjecture; if moreover $C$ is not isomorphic to the complete intersection of two cubics in $\mathbb{P}^3$, then it satisfies Green-Lazarsfeld's Gonality Conjecture as well.
\end{prop}
\begin{proof}
We assume that $C$ has genus $g\geq 4$, Clifford dimension $1$, Clifford index $c$ and gonality $k\leq(g+2)/2$. Having fixed $k\leq d\leq g-k+2$, Corollary \ref{espresso} and Proposition \ref{conto} imply that every dominating component $\W$ of $\W^1_d(\vert L\vert)$ has dimension $\leq\dim\vert L\vert+d-k$, hence $C$ satisfies the linear growth condition (\ref{mmm}). Green's Conjecture for smooth plane curve and complete intersection of two cubics in $\mathbb{P}^3$ was established by Loose in \cite{loose}, while Aprodu proved Green-Lazarsfeld's Gonality Conjecture for curves of Clifford dimension $2$ in \cite{marian1}.
\end{proof}

We proceed with the proof of Theorem \ref{veloce}.
\begin{proof}[Proof of Theorem \ref{veloce}]
We can assume $g(L)\geq 4$. By Proposition \ref{general}, if $C\in\vert L\vert_s$ is general then $K_{g-c-1,1}(C,\omega_C)=0$, where $c=\Cliff(C)$. If we show that the group $K_{g-c-1,1}(C,\omega_C)$ does not depend (up to isomorphism) on the choice of $C$ in its linear system, by semicontinuity of the Clifford index, Green's Conjecture follows for any curve in $\vert L\vert_s$.

Set $N:=L\otimes \omega_S$; since $N$ is nef and big, the hypotheses of Theorem \ref{complex} are satisfied. Indeed, (\ref{sa}) and (\ref{sso}) for $q=1$ follow directly from the fact that $S$ is regular and has geometric genus $0$. We remark that this also implies that\begin{equation*}H^0(C,\omega_C)\simeq H^0(S,L\otimes \omega_S),\,\,\,\,\forall\,C\in\vert L\vert_s.\end{equation*}Equality (\ref{sso}) for $q=0$ is trivial since $\vert L\vert$ contains a smooth, irreducible curve. For $q\geq 2$, the line bundle $N^{q-1}$ is nef and big and the Kawamata-Viehweg Vanishing Theorem (cf. \cite[Theorem 4.3.1]{positivity}) implies that
$$
0=H^1(S, N^{ -(q-1)})^\vee\simeq H^1(S,(L\otimes\omega_S)^{q-1}\otimes\omega_S)=H^1(S,N^{q}\otimes L^\vee).
$$
By adjunction, for any curve $C\in\vert L\vert_s$, we obtain the following long exact sequence
\begin{eqnarray*}
\cdots&\to& K_{g-c-1,1}(S,L^\vee,L\otimes\omega_S)\to K_{g-c-1,1}(S,L\otimes\omega_S)\to K_{g-c-1,1}(C,\omega_C)\\
&\to& K_{g-c-2,2}(S,L^\vee,L\otimes\omega_S)\to\cdots.
\end{eqnarray*}
The group $ K_{g-c-1,1}(S,L^\vee,L\otimes\omega_S)$ trivially vanishes since $H^0(S,\omega_S)=0$. By the Vanishing Theorem (\ref{vanishing}) applied to $K_{g-c-2,2}(S,L^\vee,L\otimes\omega_S)$, we can conclude that 
\begin{equation}\label{iso}
K_{g-c-1,1}(S,L\otimes\omega_S)\simeq K_{g-c-1,1}(C,\omega_C),
\end{equation}
provided that $g-c-2\geq h^0(S, L\otimes\omega_S^{ 2})$. We can assume $h^0(S, L\otimes\omega_S^{ 2})\geq 2$ and we are under the hypothesis that the anticanonical system contains a pencil. Hence, $\omega_S^\vee\otimes\oo_C$ contributes to the Clifford index and, if $C\in\vert L\vert_s$ is general, then:
\begin{align}\label{bla}
c=\Cliff(C)\leq \Cliff(\omega_S^\vee\otimes\oo_C)&=-c_1(L)\cdot K_S-2h^0(C,\omega_S^\vee\otimes\oo_C)+2\\
&=-c_1(L)\cdot K_S-2h^0(S,\omega_S^\vee)+2.\nonumber
\end{align}
 Since $H^1(S,L\otimes \omega_S^{ 2})\simeq H^1(S,L^\vee\otimes\omega_S^\vee)^\vee=0$, we have
 \begin{eqnarray*}
 h^0(S, L\otimes\omega_S^{ 2})=\chi(S,L\otimes\omega_S^{ 2})&=&g+c_1(L)\cdot K_S+K_S^2\\
 &\leq& g-c+1-h^0(S,\omega_S^\vee)-h^1(S,\omega_S^\vee).
 \end{eqnarray*}
The conclusion is straightforward unless $\chi(S,\omega_S^\vee)=h^0(S,\omega_S^\vee)=2$; in this case the hypothesis that $\Cliff(\omega_S^\vee\otimes\oo_C)>c$ for a general $C\in\vert L\vert_s$ is necessary in order to get to the conclusion. 
\end{proof}
Finally, we prove Theorem \ref{cli}.
\begin{proof}[Proof of Theorem \ref{cli}]
By Proposition \ref{cliffdim}, we can assume that general curves in $\vert L\vert_s$ have Clifford dimension $1$; we denote by $c$ their Clifford index. 

The isomorphism (\ref{iso}), valid for any curve $C\in \vert L\vert_s$, together with Green and Lazarsfeld's result stating that $K_{p,1}(C,\omega_C)\neq 0$ for $p\leq g-\Cliff(C)-2$, implies the constancy of the Clifford index. By semicontinuity of the gonality, all curves in $\vert L\vert_s$ have Clifford dimension $1$ and the same gonality.
\end{proof}

\begin{rem}\label{finale}
Knutsen \cite{leopold} proved that, if a line bundle $L$ on a Del Pezzo surface $S$ satisfies $g(L)\geq 4$, then the Clifford index curves in $\vert L\vert_s$ is constant unless $S$ has degree $1$, the line bundle $L$ is ample, $c_1(L)\cdot E\geq 2$ for every $(-1)$-curve $E$ if $c_1(L)^2\geq 8$, and there is an integer $k\geq 3$ such that $-c_1(L)\cdot K_S=k$, $c_1(L)^2\geq 5k-8\geq 7$ and $c_1(L)\cdot\Gamma\geq k$ for every smooth rational curve such that $\Gamma^2=0$. In this case, the curves through the base point of $\omega_s^\vee$ form a family of codimension $1$ in $\vert L\vert_s$, have gonality $k-1$ and Clifford index $k-3$, while a general curve $C\in \vert L\vert_s$ has gonality $k$ and Clifford index $k-2$; in particular, $\omega_S^\vee\otimes\oo_C$ computes $\Cliff(C)$. The easiest example where the gonality is not constant  is provided by $L=\omega_S^{-n}$ for $n\geq 3$.

Vice versa, if $S$ has degree $1$ and $\Cliff(\omega_S^\vee\otimes\oo_C)=\Cliff(C)$ for a general $C\in\vert L\vert_s$, one recovers Knutsen's conditions because, if $\Gamma$ is a smooth rational curve with $\Gamma^2=0$, then $\oo_S(\Gamma)$ cuts out a base point free pencil on $C$ and, if $c_1(L)^2\geq 8$ and $E$ is a $(-1)$-curve, then $\oo_C(-K_S+E)$ defines a net on $C$ which contributes to its Clifford index. This shows that the extra hypothesis we make when $\chi(S,\omega_S^\vee)=h^0(S,\omega_S^\vee)=2$ is unavoidable.
\end{rem}

\vfill


\begin{thebibliography}{CoMR}


\bibitem[A1]{marian1}
M. Aprodu, {\em On the vanishing of higher syzygies of curves}, Math. Zeit. \textbf{241} (2002), 1-15.

\bibitem[A2]{marian}
 M. Aprodu, {\em Remarks on syzygies of d-gonal curves}, Math. Res. Lett. \textbf{12} (2005), 387-400.
 
 \bibitem[A3]{marian2}
 M. Aprodu, {\em Lazarsfeld-Mukai bundles and applications}, \texttt{ArXiv:1205.4415}.

\bibitem[AF]{aprodu} 
M. Aprodu, G. Farkas, {\em The Green Conjecture for smooth curves lying on arbitrary $K3$ surfaces}, Compos. Math. \textbf{147} (2011), 839-851.

\bibitem[ACG]{arbarello2}
E. Arbarello, M. Cornalba, P. A. Griffiths, {\em Geometry of algebraic curves. Volume II. With a contribution by Joseph Daniel Harris}, Grundlehren der mathematischen Wissenschaften, 267, Springer-Verlag, Berlin (2011).


\bibitem[CM]{coppi}
M. Coppens, G. Martens, {\em Secant spaces and Clifford's Theorem}, Compos. Math. \textbf{78} (1991), 193-212.

\bibitem[CoMR]{costa}
L. Costa, R. M. Mir\'o-Roig, {\em Rationality of moduli spaces of vector bundles on rational surfaces}, Nagoya Math. J. \textbf{165} (2002), 43-69.

\bibitem[E]{ehbauer}
S. Ehbauer, {\em Syzygies of points in projective space and applications},  Zero-dimensional schemes (\mbox{Ravello}, 1992), De Gruyter, Berlin, 1994, 145-170.

\bibitem[ELMS]{cliff}
D. Eisenbud, H. Lange, G. Martens, and F.-O. Schreyer, {\em The Clifford dimension of a projective curve}, Compos. Math. \textbf{72} (1989), 173-204. 

\bibitem[G]{koszul}
M. L. Green, {\em Koszul cohomology and the geometry of projective varieties}, J. Diff. Geom. \textbf{19} (1984), 125-171.

\bibitem[Ha]{birational}
B. Harbourne, {\em Birational morphisms of rational surfaces}, J. Algebra \textbf{190} (1997), 145-162.

\bibitem[HR]{ramanan}
A. Hirschowitz, S. Ramanan, {\em New evidence for Green's conjecture on syzygies of canonical curves}, Ann. Sci. \'Ec. Norm. Sup\'er. (4) \textbf{31} (1998), 145-152.

\bibitem[K]{leopold}
A. L. Knutsen, {\em Exceptional curves on Del Pezzo surfaces}, Math. Nachr. \textbf{256} (2003), 58-81.

\bibitem[La1]{lazarsfeld}
R. Lazarsfeld, {\em Brill-Noether-Petri without degenerations}, J. Diff. Geom. \textbf{23} (1986), 299-307.

\bibitem[La2]{positivity}
R. Lazarsfeld, {\em Positivity in algebraic geometry. I}, Ergebnisse der Mathematik und ihrer Grenzgebiete, 3. Folge, 49, Springer-Verlag, Berlin (2004).

\bibitem[LC]{ciotola}
M. Lelli-Chiesa, {\em Stability of rank-3 Lazarsfeld-Mukai bundles on $K3$ surfaces}, \texttt{ArXiv:1112.2938}.

\bibitem[Lo]{loose}
F. Loose, {\em On the graded Betti numbers of plane algebraic curves}, Manuscr. Math. \textbf{64} (1989), 503-514.

\bibitem[Ma]{marti}
G. Martens, {\em \"Uber den Clifford-Index algebraischer Kurven}, J. Reine Angew. Math. \textbf{336} (1982), 83-90. 

\bibitem[P1]{delpezzo}
G. Pareschi, {\em Exceptional linear systems on curves on Del Pezzo surfaces}, Math. Ann. \textbf{291} (1991), 17-38.

\bibitem[P2]{pareschi}
G. Pareschi, {\em  A proof of Lazarsfeld's Theorem on curves on $K3$ surfaces}, J. Alg. Geom. \textbf{4} (1995), 195-200.

\bibitem[V1]{evenv}
C. Voisin, {\em Green's generic syzygy conjecture for curves of even genus lying on a $K3$ surface}, J. Eur. Math. Soc. \textbf{4} (2002), 363-404.

\bibitem[V2]{oddv}
C. Voisin, {\em Green's canonical syzygy conjecture for generic curves of odd genus}, Compos. Math. \textbf{141} (2005), 1163-1190.

\end{thebibliography}
\end{document}